\documentclass[a4paper,12pt]{article}

\usepackage{amsmath}
\usepackage{amsthm}
\usepackage{enumerate}
\usepackage{amssymb,amsfonts,latexsym,mathtools}
\usepackage{mathrsfs}
\usepackage{bm}
\usepackage{cases}
\usepackage{color}

\newcommand{\Levy}{L\'{e}vy}

\newcommand{\R}{\mathbb{R}}
\newcommand{\F}{\mathscr{F}}

\newcommand{\e}{\varepsilon}

\renewcommand{\P}{\mathbb{P}}
\usepackage{multirow}

\numberwithin{equation}{section}

\makeatletter
\renewcommand\section{\@startsection {section}{1}{\z@}%
{-3.5ex \@plus -1ex \@minus -.2ex}%
{2.3ex \@plus.2ex}%
{\normalfont\large\bf}}
\makeatother

\makeatletter
\renewcommand\subsection{\@startsection {subsection}{1}{\z@}%
{-3.5ex \@plus -1ex \@minus -.2ex}%
{2.3ex \@plus.2ex}%
{\normalfont\normalsize\bf}}
\makeatother

\theoremstyle{plain}
\newtheorem{thm}{Theorem}[section]
\newtheorem{lem}[thm]{Lemma}

\theoremstyle{definition}
\newtheorem{Rem}[thm]{Remark}
\newtheorem{Exa}[thm]{Example}
\allowdisplaybreaks

\usepackage[dvipdfmx]{hyperref}
\hypersetup{
setpagesize=false,
 bookmarksnumbered=true,
 bookmarksopen=true,
 colorlinks=true,
 linkcolor=blue,
 citecolor=red,
}

\begin{document}
\begin{center}
\Large \textbf{Hitting Probabilities of Finite Points for One-Dimensional \Levy\ Processes}
\end{center}
\begin{center}
Kohki IBA\footnote{Affiliation: Graduate School of Science, The University of Osaka, Osaka, Japan.}\textsuperscript{,}\footnote{E-mail: kohki.iba [at] gmail.com}
\end{center}
\begin{abstract}
For a one-dimensional \Levy\ process, we derive an explicit formula for the probability of first hitting a specified point among a fixed finite set. Moreover, using this formula, we obtain an explicit expression for each entry of the $Q$-matrix of the trace process on the finite set. These formulas involve solely the renormalized zero resolvent.
\end{abstract}


\section{Introduction}
Let us consider the probability that a stochastic process $X$ starting from $x$ hits $a$ before $b$. For a Brownian motion, it is well-known (see, e.g., Theorem 7.5.3 of \cite{Du}) that this probability is given as:
\begin{align}
\label{AAA1}
  \P_x(T_a<T_b)=\frac{b-x}{b-a}\qquad \text{for}\ a<x<b,
\end{align}
where $\P_x$ denotes the distribution of a Brownian motion started at $x$, and $T_a$ denotes the hitting time of $a$, that is, $T_a:=\inf \{t\ge 0;\ X_t=a\}$.

Now, we consider the multi-point generalization of this formula (\ref{AAA1}). Let $A_n:=\{a_1,...,a_n\}$ be a finite set. Let $T_{A_n}$ denote the hitting time of $A_n$, that is, $T_{A_n}:=\inf \{t\ge 0;\ X_t\in A_n\}$. We consider the hitting probability 
\begin{align}
\label{a1}
  \P_{x}(T_{a_i}=T_{A_n})
\end{align}
that the process starting from $x$ first hits $a_i$ among $A_n$. In the case of Brownian motion, by the continuity of the sample paths, it is meaningless to consider the hitting probabilities for $n\ge 3$. However, for a general point-regular \Levy\ process, the presence of jumps makes the hitting probabilities meaningful even for $n\ge 3$. In this paper, we characterize the hitting probabilities in term of the renormalized zero resolvent.

Let $(X_t)_{t\ge 0}$ be a one-dimensional \Levy\ process. We consider the following conditions:
\begin{enumerate}
  \item[] \textbf{(A)} For the characteristic exponent $\Psi$ of $(X_t)_{t\ge 0}$, it holds that
  \begin{align}
    \int_0^\infty \left|\frac{1}{q+\Psi(\lambda)}\right|d\lambda<\infty\qquad \text{for}\ q>0.
  \end{align}
  \item[] \textbf{(A1)} The process $(X_t)_{t\ge 0}$ is not a compound Poisson process.
  \item[] \textbf{(A2)} $0$ is regular for itself.
\end{enumerate}
Note that the condition \textbf{(A)} implies conditions \textbf{(A1)} and \textbf{(A2)}. See, e.g., page 7 of \cite{TY}. We recall a previous result on the probability (\ref{a1}) for \Levy\ processes in the two-point case.

\begin{thm}[Lemma 3.5 and Lemma 9.3 of Takeda-Yano \cite{TY}]
\label{TY1}
We assume that either of the following two conditions is satisfied:
\begin{enumerate}
  \item[] (a) The process $(X_t)_{t\ge 0}$ is recurrent and satisfies the condition \textbf{(A)}.
  \item[] (b) The process $(X_t)_{t\ge 0}$ is transient and satisfies conditions \textbf{(A1)} and \textbf{(A2)}.
\end{enumerate}
Then, the following hold: 
\begin{enumerate}
  \item If the condition (a) is satisfied, then 
  \begin{align}
    \P_x(T_a<T_b)=\frac{h(b-a)+h(x-b)-h(x-a)}{h(a-b)+h(b-a)},
  \end{align}
  where $h$ is the renormalized zero resolvent defined in (\ref{h}). 
  \item If the condition (b) is satisfied, then 
  \begin{align}
    \P_x(T_a<T_b)=\frac{h(b-a)+h(x-b)-h(x-a)-\kappa h(x-b)h(b-a)}{h(a-b)+h(b-a)-\kappa h(a-b)h(b-a)},
  \end{align}
  where $h$ is the renormalized zero resolvent defined in (\ref{h2}) and $\kappa$ is defined in (\ref{k}). 
\end{enumerate}
\end{thm}

We extend this result to the multi-point case and obtain the probability (\ref{a1}).

\begin{thm}
\label{A1}
We assume the hypotheses of Theorem \ref{TY1} hold. Then, it holds that
\begin{align}
\label{eq1}
{\tiny
  \begin{pmatrix}
    \P_{x}(T_{a_1}=T_{A_{n}})\\
    \P_{x}(T_{a_2}=T_{A_{n}})\\
    \vdots\\
    \P_{x}(T_{a_{n-1}}=T_{A_{n}})
  \end{pmatrix}=
  \begin{pmatrix}
    1 & \P_{a_2}(T_{a_1}<T_{a_n}) & \cdots & \P_{a_{n-1}}(T_{a_1}<T_{a_n})\\
    \P_{a_1}(T_{a_2}<T_{a_n}) & 1  & \cdots &\P_{a_{n-1}}(T_{a_2}<T_{a_n})\\
    \vdots &\vdots &&\vdots \\
    \P_{a_1}(T_{a_{n-1}}<T_{a_n}) & \P_{a_2}(T_{a_{n-1}}<T_{a_n}) &\cdots & 1 \\
  \end{pmatrix}^{-1}
  \begin{pmatrix}
    \P_x(T_{a_1}<T_{a_n})\\
    \P_x(T_{a_2}<T_{a_n})\\
    \vdots\\
    \P_x(T_{a_{n-1}}<T_{a_n})\\
  \end{pmatrix}}.
\end{align}
Here, the above matrix is the $(n-1)\times (n-1)$ matrix obtained by deleting the $n$-th row and column from $(\P_{a_l}(T_{a_k}<T_{a_n}))_{k,l=1}^n$.
\end{thm}

The proof of this theorem will be given in Section \ref{S3}.

\begin{Rem}
\label{A2}
 By Theorem \ref{TY1}, the hitting probabilities $\P_x(T_{a_i}=T_{A_n})$ for $i=1,...,n$ can be represented solely by the renormalized zero resolvent $h$.
\end{Rem}

For a general Markov process, it is known that the probability (\ref{a1}) can be expressed in terms of the $Q$-matrix of the trace process of $X$ on $A_n$. In fact, for $i=1,...,n$, let $(L_t^{a_i})_{t\ge 0}$ denote the local time process of $a_i$ for $(X_t)_{t\ge 0}$ (of course, we impose suitable conditions on the process $X$ so that the local time exists) and let $A_t=\sum_{i=1}^n L_t^{a_i}$. Since $(A_t)_{t\ge 0}$ is a continuous additive functional, we can consider its right-continuous inverse $(A_u^{-1})_{u\ge 0}.$ We define
\begin{align}
  Y_u:=X_{A_u^{-1}}\qquad \text{for}\ u\ge 0.
\end{align}
It is known that this process, which is referred to as a \emph{trace process} or an \emph{embedded process}, is a continuous-time Markov chain on the finite state space $\{a_1,...,a_n\}$. If the process $(Y_u)_{u\ge 0}$ has a $Q$-matrix $\bm{Q}=(q_{i,j})_{i,j=1}^n$, then it holds that
\begin{align}
\label{a2}
  \P_{a_i}(T_{a_j}=T_{A_n\setminus \{a_i\}})=-\frac{q_{i,j}}{q_{i,i}}.
\end{align}
See, e.g., Proposition 2.8 of Section 1.2 of \cite{And}.

We assume the hypotheses of Theorem \ref{TY1}. Then, we can define a local time of $a\in \R$, so the trace process $(Y_u)_{u\ge 0}$ is also defined. Moreover, since the \Levy\ process $(X_t)_{t\ge 0}$ is right-continuous, so is $(Y_u)_{u\ge 0}$; hence the $Q$-matrix $\bm{Q}$ exists. By combining Theorem \ref{A1} and the equation (\ref{a2}), we obtain an explicit formula of each entry of the $Q$-matrix:

\begin{thm}
\label{A3}
We assume the hypotheses of Theorem \ref{TY1} hold. Then, it holds that
\begin{align}
\label{aa2}
  q_{i,j}=-q_{i,i}\P_{a_i}(T_{a_j}=T_{A_n\setminus \{a_i\}})
\end{align}
for $i,j=1,...,n$ with $i\neq j$ and
\begin{align}
\label{aa1}
 q_{i,i}={\tiny-\begin{pmatrix}
    1\\
    1\\
    \vdots\\
    1
  \end{pmatrix}^\top\begin{pmatrix}
    1 & \P_{a_2}(T_{a_1}<T_{a_i}) & \cdots & \P_{a_{n}}(T_{a_1}<T_{a_i})\\
    \P_{a_1}(T_{a_2}<T_{a_i}) & 1  & \cdots &\P_{a_{n}}(T_{a_2}<T_{a_i})\\
    \vdots &\vdots &&\vdots \\
    \P_{a_1}(T_{a_{n}}<T_{a_i}) & \P_{a_2}(T_{a_{n}}<T_{a_i}) &\cdots & 1 
  \end{pmatrix}^{-1}
  \begin{pmatrix}
    \bm{n}^{a_i}(T_{a_1}<T_{a_i})\\
    \bm{n}^{a_i}(T_{a_2}<T_{a_i})\\
    \vdots\\
    \bm{n}^{a_i}(T_{a_{n}}<T_{a_i})
  \end{pmatrix}}
\end{align}
for $i=1,...,n$, where $\bm{n}^{a_i}$ denotes the excursion measure away from $a_i$ under suitable normalization (see, Subsection \ref{Sub22}). Here, the above matrix is the $(n-1)\times (n-1)$ matrix obtained by deleting the $i$-th row and column from $(\P_{a_l}(T_{a_k}<T_{a_i}))_{k,l=1}^n$.
\end{thm}

The proof of this theorem will be given in Section \ref{S4}. In Remark \ref{GR}, we will discuss the relation of our formulas with the results of Getoor \cite{G}.

\begin{Rem}
By Remark \ref{A2} and Theorems \ref{B1} and \ref{B2}, all entries of $\bm{Q}$ can be represented solely by the renormalized zero resolvent $h$.
\end{Rem}

\subsection{Organization}
This paper is organized as follows. In Section \ref{S2}, we prepare some general results of \Levy\ processes related to the renormalized zero resolvent. In Section \ref{S3}, we discuss the hitting probabilities and prove Theorem \ref{A1}. In Section \ref{S4}, we discuss the $Q$-matrix of the trace process and prove Theorem \ref{A3}. We also discuss the relation of our formulas with the results of Getoor \cite{G}. Finally, in Section \ref{S5}, we discuss several examples.


\section{Preliminaries}
\label{S2}
\subsection{\Levy\ Process and Resolvent Density}
Let $((X_t)_{t\ge 0},(\P_x)_{x\in \R})$ be the canonical representation of a one-dimensional \Levy\ process, where the subscript $x\in \R$ of $\P_x$ stands for the starting point. For $t\ge 0$, we denote by $\F_t^X:=\sigma(X_s,\ 0\le s\le t)$ the natural filtration of $(X_t)_{t\ge 0}$, and write $\F_t:=\bigcap_{\e>0}\F_{t+\e}^X$. For a set $A\subset \R$, let $T_A$ be the first hitting time of $A$ for $(X_t)_{t\ge 0}$, i.e.,
\begin{align}
  T_{A}:=\inf \{t\ge 0;\ X_t\in A\}.
\end{align}
For simplicity, we denote $T_{\{a\}}$ as $T_a$ for $a\in \R.$ For $\lambda\in \R$, we denote by $\Psi(\lambda)$ the characteristic exponent of $(X_t)_{t\ge 0}$, i.e., $\Psi(\lambda)$ satisfies
\begin{align}
  \P_0\left[e^{i\lambda X_t}\right]=e^{-t\Psi(\lambda)}\qquad \text{for}\ t\ge 0,
\end{align}
where $\P_x[\cdot]$ denotes the expectation under $\P_x$.

We consider the following conditions:
\begin{enumerate}
  \item[] \textbf{(A)} For the characteristic exponent $\Psi$ of $(X_t)_{t\ge 0}$, it holds that
  \begin{align}
    \int_0^\infty \left|\frac{1}{q+\Psi(\lambda)}\right|d\lambda<\infty\qquad \text{for}\ q>0.
  \end{align}
  \item[] \textbf{(A1)} The process $(X_t)_{t\ge 0}$ is not a compound Poisson process.
  \item[] \textbf{(A2)} $0$ is regular for itself.
\end{enumerate}
Note that the condition \textbf{(A)} implies conditions \textbf{(A1)} and \textbf{(A2)}. See, e.g., page 7 of \cite{TY}. Under conditions \textbf{(A1)} and \textbf{(A2)}, it is known that $(X_t)_{t\ge 0}$ has a bounded continuous resolvent density $r_q(x)$. It satisfies
\begin{align}
  \int_{\R}f(y)r_q(y)dy=\P_0\left[\int_0^\infty e^{-qt}f(X_t)dt\right]\qquad \text{for}\ q>0\ \text{and}\ f\ge 0.
\end{align}
See, e.g., Theorems II.16 and II.19 of \cite{Ber}. Moreover, it is known that $r_q(0)$ is divergent as $q\to 0+$ if and only if the process $(X_t)_{t\ge 0}$ is recurrent. See, e.g. Corollary 15.1 of \cite{Tukada} and Theorem 37.5 of \cite{Sato}.

\subsection{Local Time and Excursion}
\label{Sub22}
Assume the conditions \textbf{(A1)} and \textbf{(A2)} are satisfied. Then, we can define a local time of $a\in \R$, which we denote by $(L_t^a)_{t\ge 0}$. We can choose $(L_t^a)_{t\ge 0}$ so that it satisfies
\begin{align}
\label{B3}
  \P_x\left[\int_0^\infty e^{-qt}dL_t^a\right]=r_q(a-x)\qquad \text{for}\ q>0\ \text{and}\ x\in \R .
\end{align}
See, e.g., Section V of \cite{Ber}. Let $(\eta_u^a)_{u\ge 0}$ denote the right-continuous inverse of $(L_t^a)_{t\ge 0}$, that is,
\begin{align}
  \eta_u^a:=\inf\{t>0;\ L_t^a>u\}.
\end{align}
Then, the process $((\eta_u^a),\P_a)$ is a possibly killed subordinator. See, e.g., Proposition V.4 of \cite{Ber}.

Next, let $\bm{n}^{a}$ denote the characteristic measure of excursions away from $a\in \R$. We can choose $\bm{n}^a$ so that the subordinator $(\eta_u^a)_{u\ge 0}$ has no drift and its \Levy\ measure is $\bm{n}^a(T_a\in dt)$, that is,
\begin{align}
\label{exc1}
  \P_a[e^{-q\eta_l^a}]=\exp \Big(-l\bm{n}^a[1-e^{-qT_a}]\Big)\qquad \text{for}\ q>0,\ l\ge 0.
\end{align}

For more details on excursion theory, see, e.g., Section IV of \cite{Ber} or Section 6 of \cite{Kyp}.

\subsection{Renormalized Zero Resolvent}
\label{S23}
\subsubsection{Recurrent Case.}
Assume that the condition \textbf{(A)} is satisfied and that the process $(X_t)_{t\ge 0}$ is recurrent. Note that in this case $r_q(0)$ diverges as $q\to 0+$. For any $x\in \R$, the limit
\begin{align}
\label{h}
  h(x):=\lim_{q\to 0+}\Big(r_q(0)-r_q(-x)\Big)
\end{align}
exists and is finite. See Theorem 1.1 of Takeda-Yano  \cite{TY}. We call the limit function $h(x)$ the \emph{renormalized zero resolvent}. It is known that $h(x)$ is non-negative, continuous, and a subadditive function. See Theorem 1.1 of Takeda-Yano \cite{TY}. 

This function has been explicitly determined for several specific processes.

\begin{Exa}[Brownian motion]
\label{Exa1}
  Assume $(X_t)_{t\ge 0}$ is a Brownian motion. Then, we have
  \begin{align}
    h(x)=|x|.
  \end{align}
  See, e.g. Example 5.1 of \cite{Pa}.
\end{Exa}

\begin{Exa}[Strictly $\alpha$-stable process]
\label{Exa2}
Assume that $(X_t)_{t\ge 0}$ is a strictly $\alpha$-stable process of index $\alpha\in (1,2)$ with the \Levy\ measure
\begin{align}
  \nu(dx)=\begin{cases}
    c_+|x|^{-\alpha-1}dx &\text{on}\ (0,\infty),\\
    c_-|x|^{-\alpha-1}dx &\text{on}\ (-\infty,0),
  \end{cases}
\end{align}
where $c_+,c_-\ge 0$ and $c_++c_->0.$ Then, we have
\begin{align}
  h(x)=\frac{1}{K(\alpha)}(1-\beta\ \mathrm{sgn}(x))|x|^{\alpha-1},
\end{align}
where
\begin{align}
  \beta:=\frac{c_+-c_-}{c_++c_-},\qquad K(\alpha):=-\frac{(c_++c_-)\pi}{\alpha \tan (\frac{\pi\alpha}{2})}\left(1+\beta^2 \tan^2\left(\frac{\pi\alpha}{2}\right)\right).
\end{align}
See, Section 5 of Yano \cite{Yano2}.
\end{Exa}

\begin{Exa}[Spectrally negative \Levy\ process]
\label{Exa3}
Assume $(X_t)_{t\ge 0}$ is a spectrally negative \Levy\ process. Then, we have
\begin{align}
  h(x)=\begin{cases}
    W(x)+\frac{1-e^{\Phi(0)x}}{\Psi'(\Phi(0)+)}&\text{if $(X_t)_{t\ge 0}$ drifts to $-\infty$},\\
    W(x)-\frac{x}{\P_0[X_1^2]}&\text{if $(X_t)_{t\ge 0}$ oscillates},\\
    W(x)&\text{if $(X_t)_{t\ge 0}$ drifts to $\infty$},
  \end{cases}
\end{align}
where $W(x)$ denotes the $0$-scale function of $(X_t)_{t\ge 0}$ and $\Phi(0)$ denotes the largest root of the equation $\Psi(\lambda)=0$. See, Example 5.2 of Pant\'{i}	 \cite{Pa}.
\end{Exa}

Under additional assumptions, the renormalized zero resolvent has an integral representation as follows:

\begin{thm}[Theorem 15.1 of Tsukada \cite{Tukada}]
If the additional assumption 
\begin{align}
  \int_0^1 \left|\mathrm{Im}\left(\frac{\lambda}{\Psi(\lambda)}\right)\right|d\lambda<\infty
\end{align}
is satisfied, then it holds that
\begin{align}
  h(x)=\frac{1}{\pi}\int_0^\infty \mathrm{Re}\left(\frac{1-e^{i\lambda x}}{\Psi(\lambda)}\right)d\lambda.
\end{align}
In particular, if $\P_0[X_1^2]<\infty$ then the assumption of this theorem holds. See Lemma 3.2 of Takeda-Yano \cite{TY}.
\end{thm}

The expectation of the local time evaluated at the hitting time can be expressed in terms of the renormalized zero resolvent as follows:

\begin{thm}[Lemma 3.5 and Theorem 3.8 of Takeda-Yano \cite{TY}]
\label{B1}
We obtain the following formula:
\begin{align}
  h^B(a)&:=\P_0[L_{T_a}^0]=\frac{1}{\bm{n}^0(T_a<\infty)}=h(a)+h(-a).
\end{align}
\end{thm}

\subsubsection{Transient Case}
Assume that the conditions \textbf{(A1)} and \textbf{(A2)} are satisfied and that the process $(X_t)_{t\ge 0}$ is transient. In this case, it is known that
\begin{align}
\label{k}
  \kappa:=\lim_{q\to 0+}\frac{1}{r_q(0)}=\bm{n}^0(T_0=\infty)>0.
\end{align}
See, e.g. the equation (2.5) of \cite{TY}. For any $x\in \R$, the limit
\begin{align}
\label{h2}
  h(x):=\lim_{q\to 0+}\Big(r_q(0)-r_q(-x)\Big)=\frac{1}{\kappa}\P_x(T_0=\infty)
\end{align}
exists and is finite. See Theorem 9.1 of Takeda-Yano \cite{TY}. We call the limiting function $h(x)$ the \emph{renormalized zero resolvent}.

The expectation of the local time evaluated at the hitting time can be expressed in terms of the renormalized zero resolvent as follows:

\begin{thm}[Lemma 9.3 and Theorem 9.5 of Takeda-Yano \cite{TY}]
\label{B2}
We obtain the following formulas:
\begin{align}
  h^B(a)&:=\P_0[L_{T_a}^0]=h(a)+h(-a)-\kappa h(a)h(-a),\\
  \bm{n}^0(T_a<T_0)&=\frac{1-\kappa h(-a)}{h^B(a)}.
\end{align}
\end{thm}


\section{Hitting Probabilities of Finite Points}
\label{S3}
In this section, we prove Theorem \ref{A1}.

Before proving Theorem \ref{A1}, we present two lemmas in later use. The following lemma is a well-known fact in the theory of operator semigroups, but we provide a proof for the sake of completeness.

\begin{lem}
\label{c0}
Finite-state continuous-time Markov chains with right-continuous sample paths and almost surely finite lifetime have an invertible Green matrix.
\end{lem}
\begin{proof}
Let $E$ be the state space, let $Y$ be the Markov chain, and let $P(t)=(P_t(x,y))_{x,y\in E}$ be the transition function. Then, the $Q$-matrix $Q$ satisfies the following well-known formula:
\begin{align}
  \frac{d}{dt}P(t)=P(t)Q=QP(t).
\end{align}
If we denote $\zeta$ as the lifetime, then we have
\begin{align}
  \lim_{t\to \infty}P_t(x,y)= \lim_{t\to \infty}\P_x(Y_t=y)\le \lim_{t\to \infty}\P_x(t<\zeta)=0.
\end{align}
Let $G=(G_{x,y})_{x,y\in E}$ denote the Green matrix. Then, we have
\begin{align}
  G_{x,y}=\P_{x}\left[\int_0^\infty 1_{\{Y_t=y\}}dt\right]=\int_0^\infty \P_{x}(Y_t=y)dt=\int_0^\infty P_t(x,y)dt.
\end{align}
Thus, we have
\begin{align}
  QG=\int_0^\infty QP(t)dt=\int_0^\infty \frac{d}{dt}P(t)dt=\lim_{t\to \infty}P(t)-P(0)=-\bm{I},
\end{align}
where $\bm{I}$ denotes the identity matrix. Therefore, Green matrix $G$ is invertible and has an inverse matrix $-Q$.
\end{proof}

Next, we show that the matrix in the equation (\ref{eq1})  is invertible.

\begin{lem}
\label{c1}
The matrix 
\begin{align}
\label{C1}
  \begin{pmatrix}
    1 & \P_{a_2}(T_{a_1}<T_{a_n}) & \cdots & \P_{a_{n-1}}(T_{a_1}<T_{a_n})\\
    \P_{a_1}(T_{a_2}<T_{a_n}) & 1  & \cdots &\P_{a_{n-1}}(T_{a_2}<T_{a_n})\\
    \vdots &\vdots &&\vdots \\
    \P_{a_1}(T_{a_{n-1}}<T_{a_n}) & \P_{a_2}(T_{a_{n-1}}<T_{a_n}) &\cdots & 1 \\
  \end{pmatrix}
\end{align}
is invertible.
\end{lem}
\begin{proof}
Assume that the condition \textbf{(A)} is satisfied and that the process $(X_t)_{t\ge 0}$ is recurrent. By Theorem \ref{B1}, we have
\begin{align*}
G_{i,k}:=\P_{a_i}\left[L_{T_{a_n}}^{a_k}\right]=\P_{a_i}(T_{a_k}<T_{a_n})\P_{a_k}\left[L_{T_{a_n}}^{a_k}\right]=\P_{a_i}(T_{a_k}<T_{a_n})h^B(a_k-a_n).
  \stepcounter{equation}\tag{\theequation}
\end{align*}
Thus, we have
\begin{align}
  (\ref{C1})={\tiny\begin{pmatrix}
    \frac{1}{h^B(a_1-a_n)} &0 &\cdots &0\\
    0&\frac{1}{h^B(a_2-a_n)}&\cdots &0\\
    \vdots&\vdots&&\vdots\\
    0&0&\cdots&\frac{1}{h^B(a_{n-1}-a_n)}
  \end{pmatrix}
  \begin{pmatrix}
    G_{1,1} & G_{2,1} & \cdots & G_{n-1,1}\\
    G_{1,2} & G_{2,2}  & \cdots & G_{n-1,2}\\
    \vdots &\vdots &&\vdots \\
    G_{1,n-1} & G_{2,n-1} &\cdots & G_{n-1,n-1} \\
  \end{pmatrix}=:\bm{D}_n \bm{G}_n}.
\end{align}
Since $\bm{D}_n$ is invertible, it then remains to show that $\bm{G}_n$ is invertible.

We define
\begin{align}
  B_t&:=L_t^{a_1}+\cdots +L_t^{a_{n-1}},\\
  Z_u&:=\tilde{X}_{B_u^{-1}},
\end{align}
where $(\tilde{X}_t)_{t\ge 0}$ is the process $(X_t)_{t\ge 0}$ killed upon hitting the point $a_n$. Since $(B_t)_{t\ge 0}$ is a continuous additive functional and $(\tilde{X}_t)_{t\ge 0}$ is right-continuous, $(Z_u)_{u\ge 0}$ is a continuous-time killed Markov chain with right-continuous sample paths on the finite state space $\{a_1,...,a_{n-1}\}$. The Green function of $(Z_u)_{u\ge 0}$ is
\begin{align*}
  \P_{a_i}\left[\int_0^\infty 1_{\{Z_u=a_k\}}du\right]&=\P_{a_i}\left[\int_0^\infty 1_{\{\tilde{X}_{B_u^{-1}}=a_k\}}du\right]=\P_{a_i}\left[\int_0^\infty 1_{\{\tilde{X}_{t}=a_k\}}dB_t\right]\\
  &=\P_{a_i}\left[\int_0^{T_{a_n}} 1_{\{{X}_{t}=a_k\}}dL_t^{a_k}\right]=\P_{a_i}\left[L_{T_{a_n}}^{a_k}\right]=G_{i,k}.
   \stepcounter{equation}\tag{\theequation}
\end{align*}
Therefore, since $\bm{G}_n$ is the Green matrix of $(Z_u)_{u\ge 0}$, Lemma \ref{c0} shows that it is invertible. The proof is complete.
\end{proof}

Finally, we prove Theorem \ref{A1}.

\begin{proof}[The proof of Theorem \ref{A1}]
For $k\neq n$, we have
  \begin{align*}
  \P_{x}(T_{a_k}=T_{A_{n}})&=\P_{x}(T_{a_k}<T_{a_n})-\P_{x}(T_{A_{n}\setminus \{a_k,a_n\}}<T_{a_k}<T_{a_n})\\
  &=\P_{x}(T_{a_k}<T_{a_n})-\sum_{\substack{i;\ i\le n-1\\ i\neq k}}\P_{x}(T_{a_i}=T_{A_{n}\setminus \{a_k,a_n\}}<T_{a_k}<T_{a_n})\\
  &=\P_{x}(T_{a_k}<T_{a_n})-\sum_{\substack{i;\ i\le n-1\\ i\neq k}}\P_{x}(T_{a_i}=T_{A_{n}})\P_{a_i}(T_{a_k}<T_{a_n}).
  \stepcounter{equation}\tag{\theequation}
\end{align*}
Thus, we obtain the following matrix equation:
\begin{align*}
  &{\tiny\begin{pmatrix}
    \P_{x}(T_{a_1}=T_{A_{n}})\\
    \P_{x}(T_{a_2}=T_{A_{n}})\\
    \vdots\\
    \P_{x}(T_{a_{n-1}}=T_{A_{n}})\\
  \end{pmatrix}=
  \begin{pmatrix}
    \P_x(T_{a_1}<T_{a_n})\\
     \P_x(T_{a_2}<T_{a_n})\\
    \vdots\\
    \P_x(T_{a_{n-1}}<T_{a_n})\\
  \end{pmatrix}}\\
  &\qquad {\tiny -\begin{pmatrix}
    0 & \P_{a_2}(T_{a_1}<T_{a_n}) & \cdots & \P_{a_{n-1}}(T_{a_1}<T_{a_n})\\
    \P_{a_1}(T_{a_2}<T_{a_n}) & 0  & \cdots &\P_{a_{n-1}}(T_{a_2}<T_{a_n})\\
    \vdots &\vdots &&\vdots \\
    \P_{a_1}(T_{a_{n-1}}<T_{a_n}) & \P_{a_2}(T_{a_{n-1}}<T_{a_n}) &\cdots & 0 \\
  \end{pmatrix}
  \begin{pmatrix}
    \P_{x}(T_{a_1}=T_{A_{n}})\\
    \P_{x}(T_{a_2}=T_{A_{n}})\\
    \vdots\\
    \P_{x}(T_{a_{n-1}}=T_{A_{n}})\\
  \end{pmatrix}}.
  \stepcounter{equation}\tag{\theequation}
\end{align*}
By Lemma \ref{c1}, this matrix equation can be solved, and the equation (\ref{eq1}) follows. This completes the proof of Theorem \ref{A1}.
\end{proof}


\subsection*{Transient case}
Even when the process $(X_{t})_{t\ge 0}$ is transient and satisfies Conditions \textbf{(A1)} and \textbf{(A2)}, the result can be proved by exactly the same argument as the above. However, since we use Theorem \ref{B2} instead of Theorem \ref{B1}, the entries of the diagonal matrix $\bm{D}_n$ appearing above are slightly different.

\section{The $Q$-matrix of the trace process}
\label{S4}
In this section, we prove Theorem \ref{A3}.

Assume that the condition \textbf{(A)} is satisfied and that the process $(X_t)_{t\ge 0}$ is recurrent. Since identity (\ref{aa2}) is exactly the equation (\ref{a2}), it suffices only to prove identity (\ref{aa1}).

Recall that we have defined
\begin{align}
  A_t&:=L_t^{a_1}+\cdots +L_t^{a_n},\\
  Y_u&:=X_{A_u^{-1}}.
\end{align}
The process $(Y_u)$ is a continuous-time Markov chain with right-continuous sample paths on the finite state space $\{a_1,...,a_n\}$. Let $\bm{Q}=(q_{i,j})$ be its $Q$-matrix and define the holding time:
\begin{align}
  H:=\inf\{u>0;\ Y_u\neq Y_0\}.
\end{align}

\begin{lem}
For $i=1,...,n$, it holds that
  \begin{align}
  \label{D1}
  q_{i,i}=-\bm{n}^{a_i}(T_{A_n\setminus \{a_i\}}<\infty),
\end{align}
where $\bm{n}^{a_i}$ is an excursion measure away from $a_i$.
\end{lem}
\begin{proof}
From the general theory of Markov chains, $H$ has an exponential distribution with rate $-q_{i,i}$. By the proof of Theorem 5.4 of Getoor \cite{G}, we have
\begin{align}
  \P_{a_i}(H>u)=\P_{a_i}(A_{H}^{-1}>A_u^{-1})=\P_{a_i}(T_{A_n\setminus \{a_i\}}>A_u^{-1})=\P_{a_i}(L_{T_{A_n\setminus \{a_i\}}}^{a_i}>u).
\end{align}
Thus, $H$ and $L_{T_{A_n\setminus \{a_i\}}}^{a_i}$ are identically distributed under $\P_{a_i}$. In the same manner as the proof of Lemma 6.3 of Takeda-Yano \cite{TY}, the random variable $L_{T_{A_n\setminus \{a_i\}}}^{a_i}$ has an exponential distribution with rate $\bm{n}^{a_i}(T_{A_n\setminus \{a_i\}}<\infty)$. Therefore, we obtain the equation (\ref{D1})
\end{proof}

Finally, we compute the right-hand side of (\ref{D1}). By reordering the indices, we may assume $i=n$. We have
\begin{align}
  \bm{n}^{a_n}(T_{A_n\setminus \{a_n\}}<\infty)&=\sum_{k=1}^{n-1}\bm{n}^{a_n}(T_{a_k}=T_{A_n})
\end{align}

\begin{lem}
\label{d1}
It holds that
  \begin{align}
  {\tiny
  \begin{pmatrix}
    \bm{n}^{a_n}(T_{a_1}=T_{A_n})\\
    \bm{n}^{a_n}(T_{a_2}=T_{A_n})\\
    \vdots\\
    \bm{n}^{a_n}(T_{a_{n-1}}=T_{A_n})
  \end{pmatrix}= \begin{pmatrix}
    1 & \P_{a_2}(T_{a_1}<T_{a_n}) & \cdots & \P_{a_{n-1}}(T_{a_1}<T_{a_n})\\
    \P_{a_1}(T_{a_2}<T_{a_n}) & 1  & \cdots &\P_{a_{n-1}}(T_{a_2}<T_{a_n})\\
    \vdots &\vdots &&\vdots \\
    \P_{a_1}(T_{a_{n-1}}<T_{a_n}) & \P_{a_2}(T_{a_{n-1}}<T_{a_n}) &\cdots & 1 
  \end{pmatrix}^{-1}
  \begin{pmatrix}
    \bm{n}^{a_n}(T_{a_1}<T_{a_n})\\
    \bm{n}^{a_n}(T_{a_2}<T_{a_n})\\
    \vdots\\
    \bm{n}^{a_n}(T_{a_{n-1}}<T_{a_n})
  \end{pmatrix}}.
\end{align}
\normalsize  
\end{lem}

\begin{Rem}
  By Theorem \ref{B1}, the hitting measures $\bm{n}^{a_n}(T_{a_i}=T_{A_n})$ for $i=1,...,n-1$ can be represented solely by the renormalized zero resolvent $h$.
\end{Rem}

\begin{proof}[Proof of Lemma \ref{d1}]
By the strong Markov property of $\bm{n}^{a_n}$ (see, e.g., Theorem III.3.28 of \cite{Blu}), we have
\begin{align*}
  \bm{n}^{a_n}(T_{a_k}=T_{A_n})&=\bm{n}^{a_n}(T_{a_k}<T_{a_n})-\bm{n}^{a_n}(T_{A_n\setminus \{a_n,a_k\}}<T_{a_k}<T_{a_n})\\
  &=\bm{n}^{a_n}(T_{a_k}<T_{a_n})-\sum_{\substack{j;\ j\le n-1\\ j\neq k}}\bm{n}^{a_n}(T_{a_j}=T_{A_n\setminus \{a_n,a_k\}}<T_{a_k}<T_{a_n})\\
  &=\bm{n}^{a_n}(T_{a_k}<T_{a_n})-\sum_{\substack{j;\ j\le n-1\\ j\neq k}}\bm{n}^{a_n}(T_{a_j}=T_{A_n})\P_{a_j}(T_{a_k}<T_{a_n}).
  \stepcounter{equation}\tag{\theequation}
\end{align*}
Thus, we obtain the following matrix equation:
\begin{align*}
&{\tiny
\begin{pmatrix}
    \bm{n}^{a_n}(T_{a_1}=T_{A_n})\\
    \bm{n}^{a_n}(T_{a_2}=T_{A_n})\\
    \vdots\\
    \bm{n}^{a_n}(T_{a_{n-1}}=T_{A_n})
  \end{pmatrix}=
  \begin{pmatrix}
    \bm{n}^{a_n}(T_{a_1}<T_{a_n})\\
    \bm{n}^{a_n}(T_{a_2}<T_{a_n})\\
    \vdots\\
    \bm{n}^{a_n}(T_{a_{n-1}}<T_{a_n})
  \end{pmatrix}}\\
  &\qquad {\tiny-\begin{pmatrix}
    0 & \P_{a_2}(T_{a_1}<T_{a_n}) & \cdots & \P_{a_{n-1}}(T_{a_1}<T_{a_n})\\
    \P_{a_1}(T_{a_2}<T_{a_n}) & 0  & \cdots &\P_{a_{n-1}}(T_{a_2}<T_{a_n})\\
    \vdots &\vdots &&\vdots \\
    \P_{a_1}(T_{a_{n-1}}<T_{a_n}) & \P_{a_2}(T_{a_{n-1}}<T_{a_n}) &\cdots & 0 \\
  \end{pmatrix}
  \begin{pmatrix}
    \bm{n}^{a_n}(T_{a_1}=T_{A_n})\\
    \bm{n}^{a_n}(T_{a_2}=T_{A_n})\\
    \vdots\\
    \bm{n}^{a_n}(T_{a_{n-1}}=T_{A_n})
  \end{pmatrix}}.
  \stepcounter{equation}\tag{\theequation}
  \end{align*}
By Lemma \ref{c1}, this matrix equation can be solved, and hence the claim follows.
\end{proof}

This completes the proof of Theorem \ref{A3}.

\begin{Rem}
\label{GR}
Getoor \cite{G} considered a trace process which is killed at an exponential time measured in local time, and give an approximation of $\bm{Q}$ by the $Q$-matrix of the killed trace process. Assume that $(X_t)_{t\ge 0}$ is a Hunt process, and that each point $a_i$ is regular for itself.
Moreover, for $\lambda>0$, define the matrix $\bm{U}_A^\lambda=(U_A^\lambda(i,j))$ by the following:
\begin{align}
  U_A^\lambda(i,j):=\P_{a_i}\left[\int_0^\infty e^{-\lambda t} dL_t^{a_j}\right].
\end{align}
(Note that under our assumptions on \Levy\ processes, this coincides with the resolvent density $r_\lambda(a_j-a_i)$). Then the matrix $\bm{U}_A^\lambda$ is invertible. See Theorem 2.4 of Getoor \cite{G}. Let $-\bm{Q}^\lambda$ denotes the inverse matrix of $\bm{U}_A^\lambda$. In this setting, he obtained the following theorem.

\begin{thm}[Proposition 5.3 of Getoor \cite{G}]
\label{aa3}
It holds that
\begin{align}
  \lim_{\lambda\to 0+}\bm{Q}^\lambda=\bm{Q}.
\end{align}
\end{thm}

Since $-\bm{Q}^\lambda$ is the inverse matrix of $\bm{U}_A^\lambda$, each entry of $\bm{Q}^\lambda$ can be obtained using the resolvent. Hence, each entry of $\bm{Q}$ can be computed. However, evaluating the above limit is generally hard. Indeed, in Theorem 6.5 of the same paper, he applies the result to \Levy\ processes, but in order to compute the limit he imposes additional assumptions and carries out detailed, step-by-step calculations. 
\end{Rem}

\subsection*{Transient case}
Even when the process $(X_t)_{t\ge 0}$ is transient and satisfies Conditions \textbf{(A1)} and \textbf{(A2)}, the result can be proved by exactly the same argument as the above. However, note that when carrying out explicit computations, we should use Theorem \ref{B2} instead of Theorem \ref{B1}.


\section{Examples}
\label{S5}
In this section, we compute the $Q$-matrices of the trace process for the cases $n=2,3$, and provide explicit expressions for this matrix in the cases of Brownian motion, stable processes, and spectrally negative Lévy processes.

For simplicity, we assume that the process is recurrent and that $a_1<a_2<a_3$. For $n=2,3$, we denote the $Q$-matrix $\bm{Q}$ of the trace process as $\bm{Q}^{(2)}$ and $\bm{Q}^{(3)}$, respectively. By Theorem \ref{A3}, we have
\begin{align}
  \bm{Q}^{(2)}&= \frac{1}{h^B(a_1-a_2)}\begin{pmatrix}
   -1&1\\
   1&-1
  \end{pmatrix},\\
  \bm{Q}^{(3)}&=\begin{pmatrix}
    q_{1,1}^{(3)}& q_{1,2}^{(3)} &q_{1,3}^{(3)}\\
     q_{2,1}^{(3)}& q_{2,2}^{(3)} &q_{2,3}^{(3)}\\
      q_{3,1}^{(3)}& q_{3,2}^{(3)} &q_{3,3}^{(3)}
  \end{pmatrix},
\end{align}
with 
\begin{align}
    q_{1,1}^{(3)}= -\left(\frac{1+\P_{a_2}(T_{a_3}<T_{a_1})}{h^B(a_1-a_2)}+\frac{1+\P_{a_3}(T_{a_2}<T_{a_1})}{h^B(a_1-a_3)}\right),\\
    q_{2,2}^{(3)}=-\left(\frac{1+\P_{a_1}(T_{a_3}<T_{a_2})}{h^B(a_2-a_1)}+\frac{1+\P_{a_3}(T_{a_1}<T_{a_2})}{h^B(a_2-a_3)}\right),\\
    q_{3,3}^{(3)}=-\left(\frac{1+\P_{a_1}(T_{a_2}<T_{a_3})}{h^B(a_3-a_1)}+\frac{1+\P_{a_2}(T_{a_1}<T_{a_3})}{h^B(a_3-a_2)}\right),
\end{align}
and 
\begin{align}
      q_{i,j}^{(3)}=-q_{i,i}\P_{a_i}(T_{a_j}<T_{a_k})\qquad (\text{for}\ i\neq j\ \text{with}\ \{a_i,a_j,a_k\}=\{a_1,a_2,a_3\}).
\end{align}

\begin{Exa}
Assume $(X_t)_{t\ge 0}$ is a Brownian motion. Using the formulas of Example \ref{Exa1}, we have
\begin{align}
  \bm{Q}^{(2)}&= \frac{1}{2(a_2-a_1)}\begin{pmatrix}
   -1&1\\
   1&-1
  \end{pmatrix},\\
  \bm{Q}^{(3)}&= \begin{pmatrix}
    -\frac{1}{2(a_2-a_1)} & \frac{1}{2(a_2-a_1)} & 0\\
  \frac{1}{2(a_2-a_1)} & -\left(\frac{1}{2(a_2-a_1)}+\frac{1}{2(a_3-a_2)}\right) & \frac{1}{2(a_3-a_2)} \\
  0 & \frac{1}{2(a_3-a_2)} & -\frac{1}{2(a_3-a_2)}
  \end{pmatrix}.
\end{align}
\end{Exa}

\begin{Exa}
Assume $(X_t)_{t\ge 0}$ is a strictly $\alpha$-stable process. Using the formulas of Example \ref{Exa2}, we have
\begin{align}
  \bm{Q}^{(2)}&= \frac{K(\alpha)}{2(a_2-a_1)^{\alpha-1}}\begin{pmatrix}
   -1&1\\
   1&-1
  \end{pmatrix},\\
  \bm{Q}^{(3)}&=\begin{pmatrix}
    \varphi_1(A,B,C) &\varphi_2^+(A,B,C) &\varphi_3^+(A,B,C)\\
    \varphi_4^-(A,B,C) &\varphi_1(A,C,B) &\varphi_4^+(B,A,C)\\
    \varphi_3^-(B,A,C) &\varphi_2^-(B,A,C) &\varphi_1(B,A,C)\\
  \end{pmatrix}
\end{align}
where
\begin{align}
  A:=(a_2-a_1)^{\alpha-1},\qquad B:=(a_3-a_2)^{\alpha-1},\qquad C:=(a_3-a_1)^{\alpha-1},
\end{align}
and 
\begin{align}
  \varphi_1(A,B,C)&:= -\frac{K(\alpha)}{2}\cdot\frac{2A-B+2C}{AC},\\
  \varphi_2^\pm (A,B,C)&:=\frac{K(\alpha)}{4ABC}(2A-B+2C)\Big((1\pm\beta)(C-A)+(1\mp\beta)B\Big),\\
  \varphi_3^\pm (A,B,C) &:=\frac{K(\alpha)}{4ABC}(2A-B+2C)(1\pm\beta)(A+B-C),\\
  \varphi_4^\pm (A,B,C)&:= \frac{K(\alpha)}{4ABC}(2A+2B-C)\Big((1\pm \beta)(C-A)+(1\mp\beta)B\Big).
\end{align}
\end{Exa}

\begin{Exa}
Assume $(X_t)_{t\ge 0}$ is a recurrent spectrally negative \Levy\ process. Using the formulas of Example \ref{Exa3}, we have
\begin{align}
  \bm{Q}^{(2)}&= \frac{1}{W(a_2-a_1)}\begin{pmatrix}
   -1&1\\
   1&-1
  \end{pmatrix},\\
  \bm{Q}^{(3)}&=\begin{pmatrix}
  -\frac{1}{W(a_2-a_1)}&\frac{1}{W(a_2-a_1)}&0\\
  \frac{W(a_3-a_1)-W(a_2-a_1)}{W(a_2-a_1)W(a_3-a_2)} & -\frac{W(a_3-a_1)}{W(a_2-a_1)W(a_3-a_2)} &\frac{1}{W(a_3-a_2)}\\
  \frac{W(a_2-a_1)+W(a_3-a_2)-W(a_3-a_1)}{W(a_2-a_1)W(a_3-a_2)}&\frac{W(a_3-a_1)-W(a_3-a_2)}{W(a_2-a_1)W(a_3-a_2)}&-\frac{1}{W(a_3-a_2)}
  \end{pmatrix}.
\end{align}
\end{Exa}


\section*{Acknowledgements}
The author would like to thank Professor Kouji Yano for his careful guidance and great support. This work was supported by JST SPRING, Grant Number JPMJSP2138. 

\bibliographystyle{plain}

\end{document}